\documentclass[11pt, reqno, english]{amsart}  
\usepackage[utf8]{inputenc}
\usepackage[T1]{fontenc}
\usepackage{amsmath,amsthm}
\usepackage{amsfonts,amssymb}
\usepackage{url}
\usepackage{mathtools}  
\usepackage[colorlinks=true,urlcolor=blue,linkcolor=red,citecolor=magenta]{hyperref}
\usepackage{enumerate, paralist}
\usepackage{tikz-cd}
\usepackage[left=1in,right=1in,top=1in,bottom=1in]{geometry}

\theoremstyle{plain}
\newtheorem{theorem}{Theorem}
\newtheorem{lemma}[theorem]{Lemma}
\newtheorem{corollary}[theorem]{Corollary}

\theoremstyle{definition}

\newtheorem{remark}[theorem]{Remark}
\newtheorem{problem}[theorem]{Problem}

\newcommand{\R}{\mathbb{R}}

\newcommand{\F}{\mathcal{F}}
\newcommand{\M}{\mathcal{M}}
\renewcommand{\S}{\Sigma}
\newcommand{\KG}{\mathrm{KG}}

\begin{document}

\title[Embedding dimensions of simplicial complexes]{Embedding dimensions of simplicial complexes \\ on few vertices}



\author{Florian Frick}
\address[FF]{Dept.\ Math.\ Sciences, Carnegie Mellon University, Pittsburgh, PA 15213, USA \newline \indent Inst. Math., Freie Universit\"at Berlin, Arnimallee 2, 14195 Berlin, Germany}
\email{frick@cmu.edu} 

\author{Mirabel Hu} 
\address[MH]{Dept.\ Math.\ Sciences, Carnegie Mellon University, Pittsburgh, PA 15213, USA}
\email{zixinh@andrew.cmu.edu} 

\author{Verity Scheel}
\address[VS]{Dept.\ Math., Bard College, Annandale-on-Hudson, NY 12504, USA}
\email{vs5044@bard.edu}

\author{Steven Simon}
\address[SS]{Dept.\ Math., Bard College, Annandale-on-Hudson, NY 12504, USA}
\email{ssimon@bard.edu} 

\thanks{FF was supported by NSF grant DMS 1855591, NSF CAREER grant DMS 2042428, and a Sloan Research Fellowship.}


\begin{abstract}
\small
We provide a simple characterization of simplicial complexes on few vertices that embed into the $d$-sphere. Namely, a simplicial complex on $d+3$ vertices embeds into the $d$-sphere if and only if its non-faces do not form an intersecting family. As immediate consequences, we recover the classical van Kampen--Flores theorem and provide a topological extension of the Erd\H os--Ko--Rado theorem. By analogy with F\'ary's theorem for planar graphs, we show in addition that such complexes satisfy the  rigidity property that continuous and linear embeddability are equivalent. 
\end{abstract}

\date{\today}
\maketitle

\section{Introduction and Statement of Results}

Planar graphs are characterized as those without a $K_5$- or $K_{3,3}$-minor~\cite{Wa37}. By a theorem of F\'ary~\cite{Fa48}, a graph can be continuously embedded into the plane if and only if there is an embedding where every edge is a straight line segment. For higher-dimensional simplicial complexes and embeddings into~$\R^d$, the situation is much more intricate: No efficient characterization of complexes that embed into $\R^d$ exists in general, algorithmically deciding the existence of an embedding can be -- depending on the dimensions of the complex and codomain -- difficult or even impossible (see, e.g.,~\cite{MaTaWa11,  MeRiSeTa20}), and an analogue of F\'ary's theorem asserting the equivalence of topological and linear embeddability fails even in dimension three~\cite{Br83, Sch10}.

Here we show that if one restricts attention to simplicial complexes on few vertices relative to dimension, then both a simple characterization of complexes that embed into~$\R^d$ as well as the equivalence of topological and linear embeddability can be salvaged. As a preliminary observation, note that a finite complex embeds into a $d$-dimensional sphere ~$S^d$ but not ~$\R^d$ if and only if it is homeomorphic to~$S^d$. As the only simplicial complex on $d+2$ vertices that does not embed into the $d$-sphere~$S^d$ is the $(d+1)$-dimensional simplex~$\Delta_{d+1}$, we shall therefore focus on deciding whether simplicial complexes on $d+3$ vertices embed into~$S^d$. 

 To avoid possible confusion, we note that throughout the paper we shall adopt the usual convention in topological combinatorics of not distinguishing between an abstract simplicial complex $\Sigma$ (that is, a family of subsets of $[n]=\{1,2,\ldots, n\}$ which is closed under taking subsets) and its geometric realization $|\Sigma|$ viewed as a subcomplex of the  $(n-1)$-dimensional simplex $\Delta_{n-1}$. In particular, faces of $\Sigma$ are identified with their corresponding simplices in $\Delta_{n-1}$, and by a continuous map defined on $\Sigma$ we mean one which is defined on the complex's geometric realization. 
 
Fix a positive integer $n$. For a family $\F$ of subsets of~$[n]$, let $\S(\F)$ denote the simplicial complex defined by $$\S(\F) = \{ \sigma \subset [n] \ : \ \tau \notin \F \ \text{for all} \ \tau \subset \sigma\}.$$ Thus $\S(\F)$ is the inclusion-maximal simplicial complex on ~$[n]$ defined by the rule that each member of $\F$ is a non-face of the complex. Note that every simplicial complex $\S$ is of this form by letting $\F$ be the family of inclusion-minimal non-faces of $\S$. Finally, a family $\F$ of subsets of $[n]$ is called \emph{intersecting} if $\sigma \cap \tau \ne \emptyset$ for any two elements $\sigma$ and $\tau$ of $\F$. 

Our characterization of complexes on $d+3$ vertices that embed into $S^d$ is surprisingly simple:

\begin{theorem}
\label{thm:int}
	Let $\F$ be a family of subsets of~$[d+3]$. Then $\S(\F)$ embeds into $S^d$ if and only if $\F$ is not an intersecting family.
\end{theorem}

We show that Theorem~\ref{thm:int} is tight with respect to $d$. In one direction, suppose that $\F=\{\sigma_1, \sigma_2\}$ is a partition of $[d+3]$ into two non-empty sets. Then $\S(\F)=\partial \sigma_1 \ast \partial \sigma_2$ is the join of the boundaries of the respective simplices. Each $\partial \sigma_i$ is a sphere of dimension $|\sigma_i|-2$, so their join is a $d$-dimensional sphere. Thus $\S(\F)$ embeds into $S^d$ but not $S^{d-1}$. In the other direction, the family $\F=\{[d+3]\}$ is intersecting and $\Sigma(\F)=\partial \Delta_{d+2}$ is a $(d+1)$-dimensional sphere. Thus $\S(\F)$ does not embed into ~$S^d$ but does embed into ~$S^{d+1}$. 

As with the Hanani--Tutte theorem ~\cite{Ch34,Tu70} for non-planar graphs, our proof of Theorem~\ref{thm:int} in the case of non-embeddings shows that there are no \textit{almost embeddings} of $\Sigma(\F)$ into $S^d$, that is for any continuous map $f\colon \S(\F)\rightarrow S^d$ there are disjoint simplices of $\S(\F)$ whose images overlap (see Lemma~\ref{lem:non-embedding}). For example, suppose that $\F$ consists of all the $(d+2)$-subsets of $[2d+3]$, in which case $\S(\F)$ is the $d$-skeleton $\Delta_{2d+2}^{(d)}$ of the $(2d+2)$-simplex. As $\F$ is intersecting by the pigeonhole principle, Theorem~\ref{thm:int} recovers the classical van Kampen--Flores theorem ~\cite{Fl33, Va32} as an immediate corollary:

\begin{theorem} [Van Kampen--Flores theorem]
    \label{thm: VK-F} For any continuous map $f\colon\Delta_{2d+2}^{(d)} \rightarrow \R^{2d}$, there exist disjoint simplices $\sigma$ and $\tau$ of $\Delta_{2d+2}^{(d)}$ such that $f(\sigma)\cap f(\tau)\neq \emptyset$. \end{theorem}

In a different direction, we show that Theorem ~\ref{thm:int} implies an extremal criterion for embeddability of simplicial complexes on few vertices. Recall that the celebrated Erd\H os--Ko--Rado theorem~\cite{EKR61} states that any intersecting family of $k$-element subsets of $[n]$ has size at most~$\binom{n-1}{k-1}$ provided $n\ge 2k$. As we will see, combining this with Theorem ~\ref{thm:int} quickly gives the following:

\begin{theorem}
\label{thm:EKR}
	Let $\S$ be a simplicial complex on $d+3$ vertices with fewer  than $\binom{d+2}{k}$ faces of dimension~${k-1}$, where $k\leq \lfloor \frac{d+3}{2}\rfloor$. Then $\S$ embeds into~$S^d$.
	\end{theorem}

We observe that Theorem~\ref{thm:EKR}, together with Theorem~\ref{thm:int}, is sufficient to recover the  Erd\H os--Ko--Rado theorem itself. Thus Theorem ~\ref{thm:EKR} can be seen as a topological generalization of the latter. To see this, suppose that $d+3\ge 2k$ and that $\F$ is a family of $k$-element subsets of $[d+3]$. If $\F$ contains more than $~\binom{d+2}{k-1}$ subsets, then the complex $\S(\F)$ has fewer than $\binom{d+3}{k} - \binom{d+2}{k-1}=\binom{d+2}{k}$ faces of dimension~${k-1}$. Thus $\S(\F)$ embeds into ~$S^d$ by Theorem~\ref{thm:EKR}, and so by Theorem~\ref{thm:int} $\F$ cannot be an intersecting family. Let us also note that the lower bound of Theorem~\ref{thm:EKR} is sharp, precisely because  the upper bound of the Erd\H os--Ko--Rado theorem is.

Our final result shows that the analogue of
F\'ary's theorem holds for simplicial complexes on few vertices. To state this formally, we say that an embedding of a simplicial complex $\S$ into~$\R^d$ is \emph{linear} if the image of each face of $\S$ is the convex hull of the image of its vertices, and likewise that an embedding of $\S$ into $S^d$ is \emph{geodesic} if the image of each face is geodesically convex in~$S^d$, that is, for any two points in the image any shortest path (in the isotropic round metric) connecting them is also in the image. We then have the following rigidity theorem:

\begin{theorem}
\label{thm:lin}
	Let $\S$ be a simplicial complex on $d+3$ vertices. Then $\S$ embeds into $\R^d$ (respectively,~$S^d$) if and only if it embeds linearly into~$\R^d$ (respectively, geodesically into~$S^d$). 
\end{theorem} 

In fact, it will follow from Lemma~\ref{lem:embedding} that any simplicial complex on $d+3$ vertices which embeds into $S^d$ is actually a subcomplex of the boundary of a convex $(d+1)$-polytope on $d+3$ vertices inscribed into~$S^d$, that is, with all vertices lying on the unit sphere. While Mani~\cite{Ma72} showed that any triangulation of~$S^d$ on at most $d+4$ vertices is the boundary complex of a convex polytope, our construction in Lemma~\ref{lem:embedding} is nonetheless optimal in that there exist complexes on $d+4$ vertices that embed into~$S^d$ but which are not contained in any simplicial $(d+1)$-polytope; see Remark~\ref{rem:opt} for an example when $d=3$. It remains open whether Theorem~\ref{thm:lin} holds for simplicial complexes on $d+4$ or $d+5$ vertices; see Problem~\ref{prob:opt}. Moreover, we note that our result does not extend to linear embeddings of polyhedra. For instance, Barnette~\cite{Ba87} gives a simple example of a polyhedral $2$-complex on six vertices that embeds into~$\R^3$ but for which no linear embedding exists. 

We refer the reader to Matou\v sek~\cite{Ma08} for the basics about simplicial complexes and to Ziegler~\cite{Zi95} for the basics about polytopes.

\section{Proofs}

\subsection{Proof of Theorem~\ref{thm:int}}

To prove Theorem~\ref{thm:int}, we first show that any simplicial complex on $d+3$ vertices with intersecting non-faces cannot embed into a $d$-sphere.

\begin{lemma} 
\label{lem:non-embedding}
Let $\F$ be an intersecting family of subsets of $[d+3]$. For any continuous map $f\colon \S(\F)\rightarrow S^d$, there exist disjoint faces $\sigma$ and $\tau$ of $\S(\F)$ such that $f(\sigma)\cap f(\tau)\neq \emptyset$. 
\end{lemma} 

We shall provide two proofs of Lemma~\ref{lem:non-embedding}. The first proof gives Lemma~\ref{lem:non-embedding} as an easy corollary of Sarkaria's lower bound for dimensions of Euclidean embeddings via chromatic numbers of Kneser graphs~\cite{Sa90, Sa91, Ma08} (stated as Theorem~\ref{thm:Sarkaria}), while the second yields Lemma~\ref{lem:non-embedding} as a direct consequence of the more elementary Topological Radon theorem~\cite{BB79} (Theorem~\ref{thm:Radon} below). 

For the first proof, we set some preliminary notation. Given any family $\F$ of subsets of  $[n]$, we let $\KG(\F)$ denote its Kneser graph. Thus the vertices of $\KG(\F)$ are the elements of $\F$, with an edge connecting each pair of disjoint sets. As usual, we let $\chi(\KG(\F))$ denote the chromatic number of this graph. Thus $\chi(\KG(\F))\leq k$ means that the sets of $\F=\cup_{i=1}^k \F_i$ is the union of $k$ families $\F_1,\ldots, \F_k$, each of which is intersecting, and  in particular $\chi(\KG(\F))=1$ if and only if $\F$ is an intersecting family.

\vspace*{.1in}

\begin{theorem}[Sarkaria]
~\label{thm:Sarkaria} 
Let $\F$ be a family of subsets of $[d+k+2]$, and suppose that $\chi(\KG(\F))\leq k$. For any continuous map $\Sigma(\F)\rightarrow \R^d$, there exist disjoint simplices $\sigma$ and $\tau$ of $\Sigma(\F)$ such that $f(\sigma)\cap f(\tau)\neq \emptyset$. 
\end{theorem}

\begin{proof}[First proof of Lemma~\ref{lem:non-embedding}] Suppose that $\F$ is an intersecting family of subsets of $[d+3]$ and let $f\colon \Sigma(\F)\rightarrow S^d$. To apply Theorem~\ref{thm:Sarkaria}, let $\F'=\{\sigma\subset [d+4] \,\colon\, \sigma\in \F\}$, that is, consider the family $\F$ on ground set~$[d+4]$. Thus $\chi(\KG(\F'))=\chi(\KG(\F))=1$ and $\S(\F')=\S(\F)\ast \{d+4\}$. Viewing $S^d$ as the unit sphere in $\R^{d+1}$, we  define $f'=f\ast c \colon \Sigma(\F')=\Sigma(\F)\ast \{d+4\} \rightarrow S^d\ast \{0\}\subset \R^{d+1}$ as the join of $f$ and the map $c$ sending the vertex $d+4$ to the origin. Explicitly, $f'((1-t)x+t(d+4))=t\cdot f(x)$ for any $x\in \Sigma(\F)$ and any $t\in [0,1]$. By Theorem~\ref{thm:Sarkaria}, there exist disjoint faces $\sigma$ and $\tau$ in $\Sigma(\F')$ with $f'(\sigma)\cap f'(\tau)\neq \emptyset$. Since $\sigma$ and $\tau$ are disjoint, only one of them can contain $d+4$, say $d+4 \notin \sigma$. Thus $f'(\sigma)=f(\sigma) \subset S^d$ and so $f'(\sigma)\cap f'(\tau) \subset S^d$. If $\tau$ contains the vertex $d+4$, let $\tau'=\tau\setminus\{d+4\}$. As $(f')^{-1}(S^d)=\S(\F)$, we must have $f'(\sigma)\cap f'(\tau')\neq \emptyset$, and so it is no loss of generality to assume that $\tau$ is a face of~$\S(\F)$. Now both $\sigma$ and $\tau$ are faces of $\S(\F)$ and therefore $f(\sigma)\cap f(\tau)\neq \emptyset$. 
 \end{proof} 

Our second proof derives Lemma~\ref{lem:non-embedding} as a result of the following. 

\vspace*{.1in}

\begin{theorem}[Topological Radon theorem] 
~\label{thm:Radon} 
For any continuous map $f\colon\Delta_{d+1}\rightarrow \mathbb{R}^d$, there exist disjoint simplices $\sigma$ and $\tau$ of $\Delta_{d+1}$ such that $f(\sigma)\cap f(\tau)\neq \emptyset$. 
\end{theorem}

\begin{proof}[Second proof of Lemma~\ref{lem:non-embedding}] 

Given a continuous map $f\colon\Sigma(\F)\rightarrow S^d$, we extend it to a continuous map $f'\colon \Delta_{d+2}\rightarrow \mathbb{R}^{d+1}$ and apply Theorem~\ref{thm:Radon}. First, let $\Delta_{d+2}'$ denote the barycentric subdivision of $\Delta_{d+2}$. Recall that the vertices of $\Delta_{d+2}'$ are the barycenters of all faces of $\Delta_{d+2}$, and the subdivision decomposes $\Delta_{d+2}$ into interior disjoint ~$(d+2)$-dimensional simplices, each of whose vertex set consists of a barycenter from an $i$-dimensional face of $\Delta_{d+2}$ for each $0\leq i \leq d+2$. It follows that any $x$ in $\Delta_{d+2}$ can be uniquely expressed as a convex sum $x=(1-t)s+tc$ where $s$ lies in $\S(\F)$, $c$ is a convex combination of barycenters of pairwise incident non-faces of ~$\S(\F)$, and $0\le t\le 1$. We now define $f'\colon \Delta_{d+2}\rightarrow\mathbb{R}^{d+1}$ by letting $f'(s)=f(s)$ for each $s\in \Sigma(\F)$, $f'(b)=0$ for each barycenter of a non-face of $\Sigma(\F)$, and extending linearly. Thus $f'(x)=(1-t)f(s)$ for each $x=(1-t)s+tc$ as above. Here we again view $S^d$ as the unit sphere in~$\R^{d+1}$, and as before all that is needed is that $F$ continuously extends~$f$ with $(f')^{-1}(S^d)=\S(\F)$. 

By Theorem~\ref{thm:Radon}, there exist disjoint faces $\sigma$ and $\tau$  of $\Delta_{d+2}$ for which $f'(\sigma)\cap f'(\tau)\neq \emptyset$. As in the first proof of Lemma~\ref{lem:non-embedding}, we may assume that $\sigma$ and $\tau$ are inclusion-minimal faces with the property $f'(\sigma)\cap f'(\tau)\neq \emptyset$. If $\F$ is intersecting, then it is again easy to see that both $\sigma$ and $\tau$ must lie in $\S(\F)$. Indeed, $\sigma\cap\tau\neq \emptyset$ for any two non-faces $\sigma$ and $\tau$ of $\S(\F)$, while on the other hand $f'(\sigma)\cap f'(\tau)=\emptyset$ if $\sigma$ lies in $\S(\F)$ and $\tau$ does not. Thus $\sigma$ and $\tau$ are in $\S(\F)$, so $f(\sigma)\cap f(\tau)\neq \emptyset$. 
\end{proof}

To conclude the proof of Theorem~\ref{thm:int}, we show that $\S(\F)$ embeds in $S^d$ whenever $\F$ is not intersecting, and moreover that this embedding is geodesic.  To that end, recall that the \emph{matching number}~$\nu(\F)$ of an arbitrary family of subsets $\F$ of $[d]=\{1,\ldots, d\}$ is the maximum number of pairwise disjoint sets of~$\F$. In particular, $\nu(\F)=1$ means that $\F$ is intersecting. Thus the following lemma completes the proof of Theorem~\ref{thm:int}.

\begin{lemma}
\label{lem:embedding}
Let $\F$ be a family of subsets of $[d]$. Then $\S(\F)$ embeds geodesically into $S^{d-\nu(\F)-1}$. In fact, $\S(\F)$ is a subcomplex of the boundary of a convex $(d-\nu(\F))$-polytope with all $d$ vertices inscribed on the unit sphere.
\end{lemma}

Before giving the proof of Lemma~\ref{lem:embedding}, we recall the following. First, the boundary $\partial \Delta_k$ of any $k$-simplex $\Delta_k$ is a triangulation of the $(k-1)$-dimensional sphere and can be realized with geodesically convex faces, for example, by radially projecting a regular $k$-simplex inscribed into the sphere. Secondly, if $\S_1$ and $\S_2$ are complexes which geodesically embed into spheres $S^{d_1}$ and $S^{d_2}$, respectively, then their join $\Sigma_1\ast \Sigma_2$ geodesically embeds in the sphere $S^{d_1+d_2+1}$. This follows by considering the natural homeomorphism $h\colon S^{d_1}\ast S^{d_2}\rightarrow S^{d_1+d_2+1}$ which sends each formal convex sum $(1-t)x_1\oplus t x_2$ of the join to $(\cos(\frac{\pi}{2}t)x_1, \sin(\frac{\pi}{2}t)x_2)\in \R^{d_1+1}\times \R^{d_2+1}$ on the sphere. Here  $x_1\in S^{d_1}, x_2\in S^{d_2},$ and $0\leq t\leq 1$. The image of the segment connecting $x_1$ and $x_2$ in $S^{d_1}\ast S^{d_2}$ is a (distance-minimizing) arc connecting $x_1$ and $x_2$ in $S^{d_1+d_2+1}$ along a great circle, so composing the join $f_1\ast f_2\colon \Sigma_1\ast \Sigma_2\to S^{d_1}\ast S^{d_2}$ of two geodesic embeddings $f_1\colon \Sigma_1\to S^{d_1}$ and $f_2\colon \Sigma_2\to S^{d_2}$ together with $h$ gives a geodesic embedding of $\Sigma_1\ast \Sigma_2$ into $S^{d_1+d_2+1}$.  

\begin{proof}[Proof of Lemma~\ref{lem:embedding}] Let $\nu=\nu(\F)$, let $\M=\{S_1,\ldots, S_\nu\}$ be a maximal collection of pairwise disjoint sets of $\F$, and let $m_i=|S_i|$ for each $i\in [\nu]$. Noting that $\S(\F)$ is a subcomplex of~$\S(\M)$, we first show that $\S(\M)$ embeds into the $(d-\nu-1)$-sphere. To that end, let $\Delta_{m_i-1}$ denote the simplex determined by $S_i$. Thus each $\partial \Delta_{m_i-1}$ is a $(m_i-2)$-sphere, where $\partial \Delta_{m_i-1}=\emptyset$ if $m_i=1$. Now let $r=d-(m_1+\ldots +m_\nu)$, and denote by $\Delta_{r-1}$ the simplex determined by those vertices of $[d]$ (if any) which are not covered by~$\M$. In particular, $\Delta_{r-1}=\emptyset$ if $r=0$. As the $S_i$ are pairwise disjoint, it is easily seen that $$\S(\M)=\partial \Delta_{m_1-1}\ast\cdots \ast \partial \Delta_{m_\nu-1} \ast \Delta_{r-1}$$ is the join of the boundaries of the $\Delta_{m_i-1}$ together with the join of $\Delta_{r-1}$. Now, $\partial \Delta_{m_1-1}\ast \cdots \ast \partial \Delta_{m_\nu -1}$ is a triangulated sphere of dimension $\sum_i (m_i-2)+\nu-1=d-r-\nu-1$, the boundary of a $(d-r-\nu)$-convex polytope $P_0$ on $d-r$ vertices. Thus $\S(\M)$ and therefore $\S(\F)$ is a subcomplex of the boundary of the join $P=P_0\ast \Delta_{r-1}$, which is a convex $(d-\nu)$-polytope on $d$ vertices.

 We now show that $\S(\M)$ geodesically embeds into $S^{d-\nu-1}$. As observed in the remarks prior to the proof of Lemma~\ref{lem:embedding}, each sphere $\partial \Delta_{m_i-1}$ geodesically embeds into $S^{m_i-2}$ with its vertices inscribed on~$S^{m_i-2}$. It follows that $\partial P_0$ also geodesically embeds into $S^{d-r-\nu-1}$, and moreover with its vertices inscribed. If $r>0$, view $\Delta_{r-1}$ as a face of $\partial \Delta_r$. Thus  $\Delta_{r-1}$ also geodesically embeds into the sphere $S^{r-1}$ with its vertices inscribed, and therefore $\S(\M)=\partial P_0\ast \Delta_{r-1}$ geodesically embeds into $S^{d-\nu-1}$ with its vertices (which are the vertices of the polytope $P=P_0\ast \Delta_{r-1}$) inscribed on the sphere. As $\S(\F)$ is a subcomplex of $\S(\M)$, this completes the proof.  
\end{proof} 

\subsection{Proof of Theorem~\ref{thm:EKR}}

Using Theorem~\ref{thm:int}, we now prove our topological generalization of the Erd\H os--Ko--Rado theorem~\cite{EKR61}.

\vspace*{.1in}

\begin{theorem}[Erd\H os--Ko--Rado theorem]
~\label{EKR} 
 Let $k \ge 2$ and $n \ge 2k$ be integers. If $\F$ is an intersecting family of $k$-element subsets of~$[n]$, then $|\F| \le \binom{n-1}{k-1}$.
\end{theorem}

\begin{proof}[Proof of Theorem~\ref{thm:EKR}]
Assume that $d+3\geq 2k$ and suppose that $\S$ is a simplicial complex on $d+3$ vertices with fewer than $\binom{d+2}{k}$ faces of dimension~${k-1}$. We have that $\S=\S(\F)$, where  $\F$ is the inclusion-minimal family of non-faces of $\S$. As $\S$ has fewer than $\binom{d+2}{k}$ faces of dimension~$k-1$, $\F$ has more than $\binom{d+3}{k} - \binom{d+2}{k} = \binom{d+2}{k-1}$ subsets of order $k$. It follows from Theorem~\ref{thm:EKR} that $\F$ cannot be intersecting, and therefore $\S$ must embed  into $S^d$ by Theorem~\ref{thm:int}. 
\end{proof}

\subsection{Proof of Theorem~\ref{thm:lin}} The proof of our F\'ary-type result is again a direct consequence of the lemmas above.

\begin{proof}[Proof of Theorem~\ref{thm:lin}] Let $f\colon\S\rightarrow S^d$ be a continuous embedding. Again, we have that $\S=\S(\F)$ where $\F$ is the inclusion-minimal set of non-faces of~$\S$. By Theorem~\ref{thm:int}, $\F$ is not intersecting, and thus $\S$ geodesically embeds into $S^d$ by Lemma~\ref{lem:embedding}. To finish the proof, suppose that $f\colon \S \rightarrow \R^d$ is an embedding. Composing with the inverse of the stereographic projection  gives a continuous embedding of $\S$ into the punctured $d$-sphere. By Lemma~\ref{lem:embedding}, $\S$ is a (necessarily proper) subcomplex of the boundary of a convex $(d+1)$-polytope~$P$ and therefore linearly embeds into~$\R^d$, for example by considering the Schlegel diagram of~$P$ with respect to a facet not contained in~$\Sigma$ and projecting through this facet. 
\end{proof}

\section{Open Problems and Concluding Remarks} 

We conclude with a problem and some comments concerning the optimality of Theorem~\ref{thm:lin} and Lemma~\ref{lem:embedding}. 

\begin{problem}
\label{prob:opt}
Brehm~\cite{Br83} constructed a triangulation of the M\"obius strip on nine vertices that does not linearly embed into~$\R^3$. Thus, embeddability and linear embeddability into $\R^d$ differ for complexes on $d+6$ vertices. As far as we know, it remains open whether these notions of embeddability coincide for simplicial complexes on $d+4$ or even $d+5$ vertices.
\end{problem}

It is tempting to think that the two triangulations of $S^3$ on eight vertices that are not boundary complexes of convex polytopes~\cite{GS67, Ba73} are good candidates to show that Theorem~\ref{thm:lin} cannot be extended to complexes on $d+5$ vertices. As shown by Mihalisin and Williams~\cite{MW02}, however, these non-polytopal $3$-spheres do linearly embed into~$\R^4$. As an application of  Theorem~\ref{thm:lin}, we provide a quick alternative proof of this fact in Corollary~\ref{cor:lin-sphere} below. Nonetheless, in Remark~\ref{rem:opt} below we show that one of these two exceptional spheres can be used to show that our construction in Lemma~\ref{lem:embedding} cannot be extended to complexes on $d+4$ vertices.  We note in passing that similar difficulties arise in extending other topological properties of $d$-dimensional simplicial complexes from those on $d+3$ vertices to those on $d+4$ vertices; see~\cite{CDGO22, CDGS20} for examples involving shellability.   

\begin{remark}
\label{rem:opt}
We construct a simplicial complex $\Sigma$ on seven vertices that embeds into~$\R^3$ but which is not contained in any convex $4$-polytope. Denoting the vertex set of $\Sigma$ by $\{v, 1,2,3,4,5,6\}$, the facet list of~$\Sigma$ is given by 
\begin{align*}
    &[1,2,3,4], [3,4,5,6], [5,6,1,2], \\
    &[v,1,2,3], [v,2,3,4], [v,3,4,5], [v,4,5,6], [v,5,6,1], [v,6,1,2], \\
    &[v,1,3,5], [v,2,4,6].
\end{align*}
The eight tetrahedra around the special vertex~$v$ glue to an octahedron with $v$ in the center. The three tetrahedra in the first line attach to an annulus of six triangles around this octahedron, namely those triangles in the octahedron around $v$ that appear in the second line above. Thus $\Sigma$ is a triangulation of a $3$-ball whose boundary is an octahedron. Moreover, every possible edge is present in~$\Sigma$. Supposing that $\Sigma$ is a subcomplex of a convex $4$-polytope~$P$, consider the seven vertices of~$P$ that are vertices of~$\Sigma$. Since $\Sigma$ is simplicial, we may assume that these seven vertices are in general position in~$\R^4$ and that their convex hull~$Q$ contains~$\Sigma$ as a subcomplex. Since the boundary of $\S$ is an octahedron, the facets of $Q$ not contained in~$\S$ form a triangulation of an octahedron without additional vertices. Any triangulation of an octahedron without additional vertices must introduce one of the diagonals of the octahedron. Thus there is an edge in $Q$ which is not in~$\S$. This is a contradiction since all edges among the seven vertices are present in~$\S$.
\end{remark}

The complex $\S$ in Remark~\ref{rem:opt} is a subcomplex of ``Barnette's sphere''~$\mathcal B$ (see ~\cite{Ba73}), one of two triangulations of the $3$-sphere on eight vertices that is not the boundary complex of a convex $4$-polytope. The triangulation $\mathcal B$ is obtained from $\S$ by coning off its boundary, that is, add a new vertex~$w$, and declaring any triangle on the boundary of the $3$-ball $\S$ together with the vertex $w$ to be a new tetrahedron of~$\mathcal B$. The other nonpolytopal $3$-sphere on eight vertices is the ``Gr\"unbaum--Sreedharan sphere'' $\mathcal{GS}$; see~\cite{GS67}. Mihalisin and Williams~\cite{MW02} show\footnote{The facet list of $\mathcal B$ given in~\cite{MW02} contains a typo: The triangle $abc$ is contained in three tetrahedra.} that $\mathcal B$ and $\mathcal{GS}$ linearly embed into~$\R^4$. In the case of $\mathcal{GS}$ this already follows from the original paper~\cite{GS67} since there a diagram of $\mathcal{GS}$ is given, that is, $\mathcal{GS}$ minus a facet is linearly embedded into~$\R^3$.

As a direct consequence of Theorem~\ref{thm:lin} we derive the following corollary, which generalizes the result of Mihalisin and Williams to all dimensions:

\begin{corollary}
\label{cor:lin-sphere}
    Let $\S$ be a simplicial complex homeomorphic to $S^d$ on $d+5$ vertices with at least one missing edge. Then $\S$ admits a linear embedding into~$\R^{d+1}$.
\end{corollary}

\begin{proof}
    Let $v$ and $w$ be two vertices of $\S$ that are not connected by an edge. By removing all faces incident to $v$ or $w$ from $\S$, we obtain a complex $\S'$ on $d+3$ vertices. Since $\S$ is a $d$-sphere, $\S'$ embeds into~$\R^d$. By Theorem~\ref{thm:lin}, there is therefore a linear embedding of $\S'$ into~$\R^d$. Viewing $\R^d$ as a hyperplane in~$\R^{d+1}$, we reinsert $v$ and $w$ by placing $v$ above this hyperplane and $w$ below it. This results in a linear embedding of $\S$ into~$\R^{d+1}$, thereby completing the proof.
\end{proof}

In particular, since $v$ and $w$ are not connected by an edge in $\mathcal B$, this provides a linear embedding of $\mathcal B$ into~$\R^4$. The complex $\mathcal{GS}$ has no missing edges, so Corollary~\ref{cor:lin-sphere} does not apply directly. Nevertheless we may use Theorem~\ref{thm:lin} to construct a linear embedding of $\mathcal{GS}$ into~$\R^4$. In the notation of~\cite{GS67}, where $\mathcal{GS}$ is denoted by~$\mathcal M$, the edge $(1,3)$ is surrounded by three tetrahedra: $[1,2,3,4], [1,2,3,7], [1,3,4,7]$. The triangle $(2,4,7)$ is not present. Thus we may perform the PL-move that replaces the three tetrahedra above by $[1,2,4,7], [2,3,4,7]$. This removes the edge $(1,3)$, and deleting vertices $1$ and $3$ yields a complex that embeds linearly into~$\R^3$ by Theorem~\ref{thm:lin}. As in the proof of Corollary~\ref{cor:lin-sphere} add back these two vertices in such a way that the line connecting them intersects the triangle $(2,4,7)$. We can thus revert the PL-move to obtain a linear embedding of~$\mathcal{GS}$.

\section*{Acknowledgements} The authors thank the anonymous referees for numerous helpful comments and suggestions which improved the exposition of the paper.\\

\noindent
\textbf{Data availability statement.} Data sharing is not applicable to this article as no datasets were generated or analyzed during the current study.

\vspace*{.1in}

\noindent \textbf{Competing interests.} The authors declare that they have no conflicts of interest.

\bibliography{bib}{}
\bibliographystyle{plain}

\end{document}